\documentclass[12pt,reqno]{article}

\usepackage[usenames]{color}
\usepackage[colorlinks=true,
linkcolor=webgreen, filecolor=webbrown,
citecolor=webgreen]{hyperref}

\definecolor{webgreen}{rgb}{0,.5,0}
\definecolor{webbrown}{rgb}{.6,0,0}

\usepackage{amssymb}
\usepackage{graphicx}
\usepackage{amscd}
\usepackage{lscape}
\usepackage{tikz}
\usepackage{tikz-cd}
\usepackage{pgfplots}

\usetikzlibrary{matrix}
\usetikzlibrary{fit,shapes}
\usetikzlibrary{positioning, calc}
\tikzset{circle node/.style = {circle,inner sep=1pt,draw, fill=white},
        X node/.style = {fill=white, inner sep=1pt},
        dot node/.style = {circle, draw, inner sep=5pt}
        }
\usepackage{tkz-fct}

\usepackage{amsthm}
\newtheorem{theorem}{Theorem}

\newtheorem{proposition}[theorem]{Proposition}

\theoremstyle{definition}

\newtheorem{example}[theorem]{Example}

\usepackage{float}

\usepackage{graphics,amsmath}
\usepackage{amsfonts}
\usepackage{latexsym}
\usepackage{epsf}

\newcommand{\seqnum}[1]{\href{http://oeis.org/#1}{\underline{#1}}}

\begin{document}

\begin{center}
\vskip 1cm{\LARGE\bf A Riordan array family for some integrable lattice models} \vskip 1cm \large
Paul Barry\\
School of Science\\
South East Technological University\\
Ireland\\
\href{mailto:pbarry@wit.ie}{\tt pbarry@wit.ie}
\end{center}
\vskip .2 in

\begin{abstract} We study a family of Riordan arrays whose square symmetrizations lead to the Robbins numbers as well as numbers associated to the $20$ vertex model. We provide closed-form expressions for the elements of these arrays, and also give a canonical Catalan factorization for them. We describe a related family of Riordan arrays whose symmetrizations also lead to the same integer sequences.
\end{abstract}

\section{Introduction}

One of the outstanding mathematical stories of the last century (brilliantly recounted in \cite{Bressoud}) was that of the Robbins numbers \seqnum{A005130}, which linked the six-vertex integrable lattice model to alternating sign matrices and plane partitions, thus building a bridge between integrable models and combinatorics. The Robbins numbers $A_n$ are the sequence of integers that begins
$$1, 1, 2, 7, 42, 429, 7436, 218348, 10850216, 911835460, 129534272700, 31095744852375,\ldots$$ 
with 
$$A_n= \prod_{k=0}^{n-1} \frac{(3k+1)!}{(n+k)!}.$$ 
The corresponding numbers $B_n$ \seqnum{A358069} for the $20$-vertex integrable lattice model (with domain wall boundary conditions) \cite{DiF} begin 
$$1, 3, 23, 433, 19705, 2151843, 561696335, 349667866305, 518369549769169,\ldots.$$ Note that we have used the On-Line Encyclopedia of Integer Sequences designations \cite{SL1, SL2} for these sequences.

Our main result is the following.
\begin{proposition} We consider the square symmetrization of the Riordan array family
$$\mathfrak{R}_r=\left(\frac{1}{(1-rx)\sqrt{1-4x}}, xc(x)\right),$$ where 
$c(x)=\frac{1-\sqrt{1-4x}}{2x}$ is the generating function of the Catalan numbers. Then the principal minor sequences of $\mathfrak{R}_1$ and of $\mathfrak{R}_2$ are given by $A_{n+1}$ and $B_n$, respectively.
\end{proposition}

\section{Preliminaries on Riordan arrays}
In this section, we give a brief overview of the elements of the theory of Riordan arrays \cite{book1, book2, SGWW}.  To define Riordan arrays, we first of all set
$$\mathcal{F}_r= \{ f \in \mathcal{R}[[x]]\, |\, f(x)=\sum_{k=r}^{\infty} f_k x^k\}.$$
Here, $x$ is a ``dummy variable'' or indeterminate, and $\mathcal{R}$ is any ring in which the operations we will define make sense. Often, it can be any of the fields $\mathbb{Q}, \mathbb{R}$ or $\mathbb{C}$. When looking at combinatorial applications, it can be the ring of integers $\mathbb{Z}$ (in this case we demand that the diagonals of matrices consist of all $1$s).
We now let $g(x) \in \mathcal{F}_0$ and $f(x) \in \mathcal{F}_1$. This means that if $g(x)=\sum_{n=0}^{\infty} g_n x^n$, then $g_0 \ne 0$. Thus $g(x)$ has a multiplicative inverse $\frac{1}{g(x)} \in \mathcal{F}_0$. If $f(x)=\sum_{n=0}^{\infty}$, then we have $f_0=0$ and $f_1 \ne 0$. Then $f(x)$ will have a compositional inverse $\bar{f} \in \mathcal{F}_1$, also denoted by $f^{\langle -1 \rangle}$, which is the solution $v(x)$ of the equation $f(v)=x$ with $v(0)=0$. By definition, we have $\bar{f}(f(x))=x$ and $f(\bar{f}(x))=x$. A \emph{Riordan array} is then defined to be an element $(g, f) \in \mathcal{F}_0 \times \mathcal{F}_1$. The term ``array'' signifies the fact that every element $(g,  f)\in \mathcal{F}_0 \times \mathcal{F}_1$ has a matrix representation $(t_{n,k})_{0 \le n,k \le \infty}$ given by
$$ t_{n,k}=[x^n] g(x) f(x)^k.$$
Here, $[x^n]$ is the functional on $\mathcal{R}[[x]]$ that returns the coefficient of $x^n$ of an element in $\mathcal{R}[[x]]$.
With this notation, we can turn the set $\mathcal{F}_0 \times \mathcal{F}_1$ into a group using the product operation
$$(g(x), f(x)) \cdot (u(x), v(x)) = (g(x)u(f(x)), v(f(x)).$$
The inverse for this operation is given by
$$ (g(x), f(x))^{-1}= \left(\frac{1}{g(\bar{f}(x))}, \bar{f}(x)\right).$$
The identity element is given by $(1,x)$, which corresponds to the identity matrix.

When passing to the matrix representation, these operations correspond to matrix multiplication and taking the matrix inverse, respectively. Note that in the matrix representation, the generating functions of the columns are given by the geometric progression $g(x)f(x)^k$ in $\mathcal{R}[[x]]$. The operation of a Riordan array on a power series $h(x)=\sum_{n=0}^{\infty} h_n x^n$ is given by
$$ (g(x), f(x)) \cdot h(x) = g(x)h(f(x)).$$
This weighted composition rule is called the fundamental theorem of Riordan arrays.
In matrix terms, this is equivalent to multiplying the column vector $(h_0, h_1, h_2,\ldots)^T$ by the matrix $(t_{n,k})$.
Because $f \in \mathcal{F}_0$, Riordan arrays have lower-triangular matrix representatives. As an abuse of language, we use the term ``Riordan array'' interchangeably to denote either the pair $(g, f)$ or the matrix $(t_{n,k})$, letting the context indicate which is being referred to at the time.
The bivariate generating function of the array $(g, f)$ is given by
$$\frac{g(x)}{1-y f(x)}.$$
Thus we have
$$[x^n] g(x)f(x)^k = [x^n y^k] \frac{g(x)}{1-y f(x)}.$$
To see this, we have
\begin{align*}
[x^n y^k]\frac{g(x)}{1-y f(x)} &=
[x^n y^k]  g(x) \sum_{i=0}^{\infty} y^i f(x)^i\\
&=[x^n] g(x) [y^k]\sum_{i=0}^{\infty} y^i f(x)^i \\
&=[x^n] g(x) f(x)^k.\end{align*}
Setting $y=0$ and $y=1$, respectively, in the generating function, yields the generating functions of the row sums and the diagonal sums of the matrix $(t_{n,k})$. That is, the row sums of the Riordan array $(g(x), f(x))$ have generating function $\frac{g(x)}{1-f(x)}$, while the diagonal sums have generating function $\frac{g(x)}{1-xf(x)}$.

\section{Square symmetrization}
Given a Riordan array $(g(x), f(x))$, its generating function is the bivariate power series
$$B(x,y)=\frac{g(x)}{1-yf(x)}.$$
By the \emph{square symmetrization} of the Riordan array $(g(x),f(x))$ we mean the matrix whose generating function is given by
$$\mathfrak{S}(x,y)=B\left(xy, \frac{1}{y}\right)+B\left(xy, \frac{1}{x}\right)-g(xy).$$ 
It is clear that 
$$\mathfrak{S}(x,y)=\mathfrak{S}(y,x)$$ and hence the square symmetrization of a Riordan array is a symmetric matrix.
\section{Riordan arrays for the Robbins numbers}
Following \cite{Fib} we present two Riordan arrays whose symmetrizations lead to the Robbins numbers $A_{n+1}$.
\begin{example}
We consider the Riordan array
$$\left(\frac{1}{1+x+x^2}, \frac{x}{1+x}\right).$$ 
The generating function of this array is given by 
$$B(x,y)=\frac{\frac{1}{1+x+x^2}}{1-y\frac{x}{1+x}}.$$ 
Thus we have $$B(x,y)=\frac{1+x}{(1+x+x^2)(1+x-xy)}.$$ 
Then we have 
$$\mathfrak{S}(x,y)=\frac{1}{(1-y+xy)(1-x+xy)}.$$ 
The resulting symmetric matrix begins 
$$\left(\begin{array}{rrrrrrr}
1 & 1 & 1 & 1 & 1 &  1 & \cdots \\
1 & -1 &-2 & -3 & -4 & -5 & \cdots \\
1 & -2 & 0 & 2 & 5 & 9 & \cdots \\
1 & -3 & 2 & 1 & -1 & -6 & \cdots \\
1 & -4 & 5 & -1 &-1 & 0 & \cdots \\
1 & -5 & 9 & -6 & 0 & 0 & \cdots \\
\vdots & \vdots & \vdots & \vdots & \vdots & \vdots & \ddots
\end{array}\right).$$
Then the principal minor sequence of this matrix is given by $(-1)^{\binom{n}{2}}A_{n+1}$. 
\end{example}
\begin{example}
We next consider the Riordan array $\left(\frac{1}{(1-x)\sqrt{1-4x}}, xc(x)\right)$ of our proposition.
We have 
$$B(x,y)=\frac{(y-2)\sqrt{1-4x}-4xy+y}{2(1-x)(4x-1)(1-y+xy^2)}.$$ 
Then we find that 
$$\mathfrak{S}(x,y)=\frac{1}{(1-xy)(1-x-y)}.$$ 
It is a classical result that the Robbins numbers $A_{n+1}$ are given by the principal minors sequence of the matrix 
$$\binom{n+k}{k}-\delta_{n,k+1}$$ which has generating function 
$$\frac{1}{1-x-y} - \frac{y}{1-xy}.$$ 
Multiplying the matrix with this generating function by the unipotent Riordan array
$$\left(\frac{1}{(1-y)(\frac{1}{1-y}-y)},y\right)$$ produces the matrix with generating function $\mathfrak{S}(x,y)$. 
Since the Riordan array $\left(\frac{1}{(1-y)(\frac{1}{1-y}-y)},y\right)$ is unipotent, all its principal minors are $1$, which establishes the result.
\end{example}

\section{The sequence $B_n$ of the $20$-vertex model}
We now consider the Riordan array
$$\left(\frac{1}{(1-2x)\sqrt{1-4x}}, xc(x)\right).$$ 
We have
$$B(x,y)=\frac{(y-2)\sqrt{1-4x}-4xy+y}{2(1-x)(4x-1)(1-y+xy^2)},$$ and we find that this simplifies to give
$$\mathfrak{S}(x,y)=\frac{1}{(1-2xy)(1-x-y)}.$$ 
We must show that the principal minor sequence of the matrix with this generating function is $B_n$. From \cite{DiF}, we know that $B_n$ is given by the principal minor sequence of the matrix with generating function
$$\frac{2y}{(1-y)(1-x-y-xy)}+\frac{1}{1-xy}=\frac{(1-x)(1+y^2)}{(1-y)(1-xy)(1-x-xy)}.$$ 
Using the fundamental theorem of Riordan arrays (multiplying the corresponding matrix by appropriate unipotent Riordan arrays), we obtain the symmetric matrix with generating function
$$g(x,y)=\frac{(1-x)(1-y)}{(1-xy)(1-x-y-xy)}$$ which begins
$$\left(\begin{array}{rrrrrrr}
1 & 0 & 0 & 0 & 0 &  0 & \cdots \\
0 & 3 &2 & 2 & 2 & 2 & \cdots \\
0 & 2 & 9 & 12 & 16 & 20 & \cdots \\
0 & 2 & 12 & 35 & 62 & 98 & \cdots \\
0 & 2 & 16 & 62 &161 & 320 & \cdots \\
0 & 2 & 20 & 98 & 320 &803 & \cdots \\
\vdots & \vdots & \vdots & \vdots & \vdots & \vdots & \ddots
\end{array}\right).$$
We now multiply by the Riordan array $\left(\frac{1}{1+x}, \frac{x}{1+x}\right)$ on the left, and by its transpose on the right, to get the matrix with generating function
$$\frac{1}{(1+x)(1+y)} g\left(\frac{x}{1+x}, \frac{y}{1+y}\right)
=\frac{1}{(1-2xy)(1+x+y)}.$$  Multiplying this new matrix on the left and on the right by the Riordan array matrix $(1,-x)$ now gives us the matrix with generating function
$$\frac{1}{(1-2xy)(1-x-y)}.$$ As all matrix transformations (except for the last) involve unipotent matrices, the principal minor sequences remain the same. For the last transformation, the matrix $(1,-x)$ has principal minor sequence $(-1)^{\binom{n+1}{2}}$, which when multiplied by itself yields the unit sequence as required. We conclude that the matrices with generating functions $\frac{2y}{(1-y)(1-x-y-xy)}+\frac{1}{1-xy}$ and $\frac{1}{(1-2xy)(1-x-y)}$ have the same principal minor sequence, namely $B_n$.

\section{Further observations}
For general $r$, the principal minor sequence arising from the Riordan array $\left(\frac{1}{(1-rx)\sqrt{1-4x}}, xc(x)\right)$ begins
$$1, r + 1, r^3 + 2r^2 + 3·r + 1, r^6 + 3r^5 + 7r^4 + 13r^3 + 11r^2 + 6r + 1,\ldots.$$ 
For $r=0,\ldots,5$ we exhibit the first few elements of the corresponding sequences in matrix form.
$$\left(\begin{array}{rrrrrrr}
1 & 1 & 1 & 1 & 1 & 1 & \cdots \\
1 & 2 & 7 &   42 &    429 &       7436 & \cdots \\
1 & 3 & 23 &  433 &   19705 &     2151843 & \cdots \\
1 & 4 & 55 &  2494 &  365953 &    171944344 & \cdots \\
1 & 5 & 109 & 9993 &  3791001 &    5898286349 & \cdots \\
1 & 6 & 191 & 31306 &  26094301 &   109913708076 & \cdots \\
\vdots & \vdots & \vdots & \vdots & \vdots & \vdots & \ddots
\end{array}\right).$$
We have 
$$\left(\frac{1}{(1-rx)\sqrt{1-4x}}, xc(x)\right)^{-1}=\left((1-2x)(1-rx+rx^2), x(1-x)\right).$$
For $r=0$ and $r=1$, the principal minor sequences of the symmetrization of this matrix begin, respectively,
$$1, -3, -13, 81, 144, -2017, -1757, 79513, 22704,\ldots$$ and 
$$1, -4, -33, 427, 5046, -56241, -316626, 7178034, 26671624,\ldots.$$ 
We are not aware of any combinatorial interpretation of these numbers.

\section{A second family of Riordan arrays}
We consider the family of Riordan arrays $\tilde{\mathfrak{R}}_r$ defined by
\begin{align*}\tilde{\mathfrak{R}}_r&=\left(\frac{1}{(1-x)\sqrt{1-2(r+2)x+r^2x^2}}, \frac{1-rx-\sqrt{1-2(r+2)x+r^2x^2}}{2}\right)\\&=\left((1-x+rx+x^2)\frac{d}{dx}\left(\frac{x(1-x)}{1+x}\right), \frac{x(1-x)}{1+rx}\right)^{-1}.\end{align*}
\begin{proposition} The Robbins numbers $A_{n+1}$ and the numbers $B_n$ are given by the principal minor sequences of the symmetrizations of $\tilde{\mathfrak{R}}_0$ and  $\tilde{\mathfrak{R}}_1$, respectively.
\end{proposition}
\begin{proof} We find that the generating function of the symmetrization of $\tilde{\mathfrak{R}}_r$ is given by
$$\frac{1}{(1-xy)(1-x-y-rxy)}.$$ 
The result follows from this.
\end{proof}
The principal minor sequence of $\tilde{\mathfrak{R}}_2$ is the sequence that begins 
$$1, 4, 55,  2494,  365953,    171944344,\ldots,$$ for example. In general, the principal minor sequence of the symmetrization of $\mathfrak{R}_r$ is that of the symmetrization of $\tilde{\mathfrak{R}}_{(r-1)}$.

\section{The Riordan arrays $\mathfrak{R}_1$ and $\mathfrak{R}_2$}
The Riordan array $\mathfrak{R}_1$ is given by
$$\mathfrak{R}_1=\left(\frac{1}{(1-x)\sqrt{1-4x}},xc(x)\right)=\left((1-x)^3-x^3, x(1-x)\right).$$ It begins
$$\left(\begin{array}{rrrrrrr}
1 & 0 & 0 & 0 & 0 &  0 & \cdots \\
3 & 1 & 0 & 0 & 0 &  0& \cdots \\
9 & 4 & 1 & 0 & 0 &  0 & \cdots \\
29 & 14 & 5 & 1 & 0 & 0 & \cdots \\
99 & 49 & 20 & 6 &1 & 0 & \cdots \\
351 & 175 & 76 & 27 & 7 & 1 & \cdots \\
\vdots & \vdots & \vdots & \vdots & \vdots & \vdots & \ddots
\end{array}\right).$$
The general term of this matrix can be expressed as 
$$t_{n,k}=\sum_{j=0}^{n-k}\binom{k+2j}{j}$$ for $k \le n$, and $0$ otherwise.
It embeds into the Riordan array $\left(\frac{(1-x)^3-x^3}{(1-x)}, x(1-x)\right)^{-1}$, which is \seqnum{A361654}. The $(n,k)$-th element of this matrix  gives the number of nonempty subsets of $\{1,...,2n-1\}$ with median $n$ and minimum $k$. 

The symmetrization of $\mathfrak{R}_1$ has its general $(n,k)$-th term given by
$$s_{n,k}=[k \le n] \sum_{j=0}^k \binom{n-k+2j}{j} + [k>n] \sum_{j=0}^n \binom{k-n+2j}{j}.$$ Here, we use the Iverson bracket $[P]$ which evaluates to $1$ if $P$ is true, and $0$ otherwise \cite{Concrete}.
This symmetric matrix begins 
$$\left(\begin{array}{rrrrrrr}
1 & 1 & 1 & 1 & 1 &  1 & \cdots \\
1 & 3 & 4 & 5 & 6 &  7& \cdots \\
1 & 4 & 9 & 14 & 20 & 27 & \cdots \\
1 & 5 & 14 & 29 & 49 & 76 & \cdots \\
1 & 6 & 20 & 49 & 99 & 175 & \cdots \\
1 & 7 & 27 & 76 & 175 & 351 & \cdots \\
\vdots & \vdots & \vdots & \vdots & \vdots & \vdots & \ddots
\end{array}\right).$$

The matrix $\mathfrak{R}_2$ is given by 
$$\mathfrak{R}_2=\left(\frac{1}{(1-2x)\sqrt{1-4x}},xc(x)\right)=\left((1-x)^4-x^4, x(1-x)\right)^{-1}.$$ It begins
$$\left(\begin{array}{rrrrrrr}
1 & 0 & 0 & 0 & 0 &  0 & \cdots \\
4 & 1 & 0 & 0 & 0 &  0& \cdots \\
14 & 5 & 1 & 0 & 0 &  0 & \cdots \\
48 & 20 & 6 & 1 & 0 & 0 & \cdots \\
166 & 75 & 27 & 7 &1 & 0 & \cdots \\
584 & 276 & 110 & 35 & 8 & 1 & \cdots \\
\vdots & \vdots & \vdots & \vdots & \vdots & \vdots & \ddots
\end{array}\right).$$
Its general term $t_{n,k}$ is given by 
$$t_{n,k}=\sum_{j=0}^{n-k}2^{n-j-k}\binom{k+2j}{j}$$ for $k \le n$, and $0$ otherwise.

In general, the Riordan array $\mathfrak{R}_r$ will have general term $\sum_{j=0}^{n-k}r^{n-j-k}\binom{k+2j}{j}$ for $k \le n$, and $0$ otherwise.

We have a canonical factorization of the arrays $\mathfrak{R}_r$ in terms of the Catalan matrix $(c(x), xc(x))$ \seqnum{A033184} as follows.
$$ \mathfrak{R}_r = (c(x), xc(x))\cdot \left(\frac{1-x}{(1-2x)(1-rx+rx^2)}, x\right).$$ 

\section{Conclusions}
The significance of this note is that there exists a family of simply defined Riordan arrays whose symmetrizations appear to be closely linked to important outputs of certain integrable lattice models. On the one hand, this points to those lattice models being the related to many objects of combinatorial interest - in this case, to elements of the Riordan group. On the other hand, it shows that Riordan arrays can play important roles in unexpected and important areas seemingly far removed from their origin. Of course, the Riordan array $\left(\frac{1}{1-x}, \frac{x}{1-x}\right)$, that is, Pascal's triangle $\left(\binom{n}{k}\right)$ \seqnum{A007318}, has always played an important role in the analysis of lattice models.

\bigskip
\hrule
\bigskip
\noindent 2020 {\it Mathematics Subject Classification}: 
Primary 15B36; Secondary 05A15, 11B83, 11C20, 15A15, 82B20, 82B23
.
\noindent \emph{Keywords:} Robbins number, Riordan array, integrable lattice model, six-vertex model, twenty-vertex model, generating function.

\bigskip
\hrule
\bigskip
\noindent (Concerned with sequences
\seqnum{A007318},
\seqnum{A005130},
\seqnum{A033184},
\seqnum{A358069}, and
\seqnum{A361654}).


\begin{thebibliography}{99}



\bibitem{book1} P. Barry, \emph{Riordan Arrays: a Primer}, Logic Press, 2017.

\bibitem{Fib} P. Barry, From Fibonacci to Robbins: series reversion
and Hankel transforms, \emph{J. Integer Seq.}, \textbf{24} (2021), \href{https://cs.uwaterloo.ca/journals/JIS/VOL24/Barry2/barry461.html} {Article 21.10.2}.

\bibitem{Bressoud} D. M. Bressoud, \emph{Proofs and Confirmations}, Cambridge University Press.

\bibitem{DiF} P. Di Francesco and E. Guitter, Twenty-vertex model with domain wall boundaries and domino tilings, \emph{Electron. J. Comb.}, \textbf{27} (2020), Article Number P2.13.
  

\bibitem{Concrete} R. Graham, D. Knuth, and O. Patashnik. \emph{Concrete Mathematics: A Foundation for Computer Science},  Addison Wesley Longman Publishing Co.

\bibitem{book2} L. Shapiro, R. Sprugnoli, P. Barry, G.-S. Cheon, T.-X. He, D. Merlini, and W. Wang, \emph{The Riordan Group and Applications}, Springer, 2022.

\bibitem{SGWW} L. W. Shapiro, S. Getu, W. J. Woan, and L. C. Woodson,
The Riordan group, \emph{Discr. Appl. Math.} \textbf{34} (1991),
 229--239.

\bibitem{SL1} N. J. A.~Sloane, \emph{The
On-Line Encyclopedia of Integer Sequences}. Published electronically
at \texttt{http://oeis.org}, 2024.

\bibitem{SL2} N. J. A.~Sloane, The On-Line Encyclopedia of Integer
Sequences, \emph{Notices Amer. Math. Soc.} \textbf{50} (2003),  912--915.


\end{thebibliography}
\end{document}